\documentclass{ifacconf}

\usepackage{mathtools}
\usepackage{graphicx}      % include this line if your document contains figures
\usepackage{natbib}        % required for bibliography
\usepackage{amsmath,amssymb,amsfonts}
\usepackage{mathrsfs}
\usepackage{array}
\usepackage{graphicx}
\usepackage{textcomp}
\usepackage{xcolor}
 \usepackage{enumerate}

\usepackage{algorithm}
\usepackage{algorithmic}
\newtheorem{assumption}{Assumption}
\newtheorem{theorem}{Theorem}
\newtheorem{lemma}{Lemma}
\newtheorem{remark}{Remark}

\newcommand{\ADDBoYou}[1]{\textcolor{black}{{#1}}}
\newcommand{\ADDBoYu}[1]{\textcolor{black}{{#1}}}
\begin{document}
\begin{frontmatter}
%\title{Distributed Consensus  Mixed-Integer Optimization via Mix-CALADIN}
\title{Mix-CALADIN: A Distributed Algorithm for Consensus Mixed-Integer Optimization}
%\title{Distributed Mixed-Integer Optimization via Mix-CALADIN\thanksref{footnoteinfo}} 
% Title, preferably not more than 10 words.

% \thanks[footnoteinfo]{Sponsor and financial support acknowledgment
% goes here. Paper titles should be written in uppercase and lowercase
% letters, not all uppercase.}
 \thanks{$^\dagger$ Corresponding author}
 \thanks[footnoteinfo]{The work of B.H., X.D., and A.I.R. was supported by the Guangzhou-HKUST(GZ) Joint Funding Scheme (Grant No. 2025A03J3960). The work of A.I.R. was also supported by the Guangdong Provincial Project (Grant No. 2024QN11G109).}

\author[First]{Boyu Han}
\author[First]{Xu Du}
%\author[Second]{Ivano Notarnicola}
\author[Third]{Karl~H.~Johansson} 
\author[First,Forth]{Apostolos I. Rikos$^\dagger$}

\address[First]{The Artificial Intelligence Thrust of the Information Hub, The Hong Kong University of Science and Technology (Guangzhou), Guangzhou, China, (e-mail: \{boyuhan,michaelxudu,apostolosr\}@hkust-gz.edu.cn).}
%\address[Second]{Department of Electrical, Electronic, and Information Engineering, University of Bologna, Italy, (e-mail: ivano.notarnicola@unibo.it)}
\address[Third]{The Division of Decision and Control Systems, KTH Royal Institute of Technology, SE-100 44 Stockholm, Sweden, (e-mail: kallej@kth.se)}
\address[Forth]{The Department of Computer Science and Engineering, The Hong Kong University of Science and Technology, Clear Water Bay, Hong Kong, China.}

\begin{abstract}                % Abstract of 50--100 words
This paper addresses distributed consensus optimization problems with mixed-integer variables, with a specific focus on Boolean variables. We introduce a novel distributed algorithm that extends the Consensus Augmented Lagrangian Alternating Direction Inexact Newton (CALADIN) framework by incorporating specialized techniques for handling Boolean variables without relying on local mixed-integer solvers. Under the mild assumption of Lipschitz continuity of the objective functions, we establish rigorous convergence guarantees for both convex and nonconvex mixed-integer programming problems. Numerical experiments demonstrate that the proposed algorithm achieves competitive performance compared to existing approaches while providing rigorous convergence guarantees.
\end{abstract}


\end{frontmatter}
%===============================================================================

\section{Introduction}
Mixed-integer programming (MIP) has attracted increasing attention in recent years due to its wide range of applications, including express transport network design (\cite{zhu2025distributed}), sensor selection and scheduling (\cite{dhingra2025admm}), wireless antenna selection (\cite{zhu2022alternating}), and optimal demand shut-offs in AC microgrids \cite{du2023optimal}.
However, as the scale and complexity of optimization problems continue to grow, solving MIPs on a single computing node has become increasingly impractical (\cite{2024_doostmohammadian_rikos_Johansson_survey}).

Traditional MIP algorithms are primarily centralized and include classical methods such as branch-and-bound, cutting-plane, and branch-and-cut.
For a comprehensive overview of these techniques, readers are referred to \cite{wolsey2020integer} and \cite{conforti2014integer}.
In addition, several heuristic approaches have demonstrated strong empirical performance without formal convergence guarantees, such as the semidefinite relaxation (SDR) method proposed by \cite{luo2010semidefinite} and the $\epsilon$-approximation algorithm for mixed-integer quadratic programming (MIQP) presented by \cite{pia2023approximation}.
However, most of the aforementioned studies focus on centralized computation, which often suffers from high memory demands and long computation times in large-scale settings.
Therefore, developing distributed algorithms for solving MIPs has become an increasingly urgent and essential research direction.

Distributed optimization algorithms can be broadly categorized into two classes \cite{ling2015dlm}: (i) primal decomposition and (ii) dual decomposition.
In primal-decomposition-based methods, each agent updates its local variables by computing and exchanging primal information such as local optima, gradients, or Hessians.
In contrast, dual-decomposition-based methods not only exchange primal information but also update and transmit dual variables associated with the coupling constraints.
This paper focuses primarily on solving MIP problems using distributed optimization algorithms of the dual decomposition type.

\noindent
\textbf{Existing Literature.}    Both primal and dual decomposition approaches have been extensively studied for solving MIPs.
On the primal-decomposition side, \cite{kuwata2010cooperative} proposed a cooperative distributed algorithm for trajectory optimization using mixed-integer linear programming (MILP), which sequentially optimizes subproblems while allowing low-dimensional parametric perturbations to neighboring trajectories, thereby ensuring robust feasibility and monotonic improvement of the overall fleet objective.
Similarly, \cite{camisa2021distributed} developed a primal-decomposition-based distributed framework for MILPs that guarantees feasible solutions within finite time and achieves low sub-optimality bounds.
In \cite{testa2019distributed}, agents iteratively generated cutting planes and exchanged active constraints to achieve convergence to a feasible solution.
However, the aforementioned works are limited to MILP problems.
More recently, the Partially Distributed Optimization Algorithm (PaDOA) proposed by \cite{murray2021partially} extended the primal decomposition framework to mixed-integer convex programs (MICP) by alternating between solving decoupled mixed-integer subproblems and a centralized MILP master problem, achieving global optimality within an $\epsilon$-convergence accuracy.
Nevertheless, this method still relies on the integration of existing centralized MILP solvers to complete its computation.

% Primal-based Distributed: 
% \begin{itemize}
%     \item Trajectory Optimization Using MILP \cite{kuwata2010cooperative} proposed a primal-based cooperative distributed algorithm that sequentially optimizes subproblems while permitting low-dimensional parametric perturbations to neighbors' trajectories, ensuring robust feasibility and monotonic improvement of the fleet objective.  \cite{camisa2021distributed} primal-based for MILP.
%     \item \cite{testa2019distributed} proposed a primal-based distributed algorithm for mixed-integer linear programs, where agents iteratively generate {\color{blue}{cutting planes}} and exchange active constraints to converge to a solution. 
%     \item \cite{murray2021partially} The PaDOA algorithm is a primal distributed optimization method that alternates between solving decoupled mixed-integer subproblems and a central MILP master problem to achieve global optimality in structured mixed-integer convex programs
%     {\color{blue}{$\epsilon$ convergence accuracy.}}
% \end{itemize}

To the best of our knowledge, dual-decomposition-based distributed algorithms for mixed-integer programming have received even greater research attention.
Within the dual decomposition framework, \cite{vujanic2016decomposition} and \cite{falsone2019decentralized} proposed dual-based decentralized algorithms for multi-agent MILPs, employing an adaptive constraint-tightening scheme to achieve finite-time convergence to feasible solutions.
Building upon this line of work, \cite{yfantis2023hierarchical} integrated second-order updates on the dual side, resulting in significantly faster convergence.
More recently, \cite{dong2024distributed} extended the applicability of distributed mixed-integer optimization to mixed-integer nonlinear programming (MINLP) problems, including nonconvex cases, by introducing a convexification procedure for local subproblems.
However, all of the above methods require each agent to invoke external MIP solvers for subproblem computations, which limits scalability and real-time applicability.

The Alternating Direction Method of Multipliers (ADMM), as a representative of dual decomposition methods, has been extensively studied for distributed mixed-integer optimization. In \cite{sun2024decomposition}, ADMM was applied to solve MILPs with a rigorous convergence analysis. Similarly, \cite{zhu2025distributed} employed the ADMM framework to tackle the express transport network design problem formulated as an MILP. In contrast, \cite{alavian2017improving} extended ADMM to MICP, where the update of integer variables relies on a centralized MIQP solver. Notably, \cite{liu2022distributed} demonstrated that, under sufficiently large augmented Lagrangian parameters, ADMM can achieve a zero duality gap for MICP, representing one of the strongest convergence results for ADMM in this context. However, the aforementioned methods all require agents or a central coordinator to solve MIP subproblems, which can be computationally expensive. Only a few works have avoided the use of centralized MIP solvers. For instance, \cite{le2025distributed} addressed an MIQP traffic coordination problem using ADMM combined with a \emph{Big-M} method, but the approach exhibited numerical instability. Similarly, \cite{takapoui2017alternating}, \cite{takapoui2020simple} and \cite[Eq.~(3)]{alavian2017improving} directly rounded integer variables within ADMM to ensure convergence; however, these methods are heuristic in nature and do not provide formal convergence guarantees. %, \cite{takapoui2017alternating},

The Augmented Lagrangian based Alternating Direction Inexact Newton (ALADIN) algorithm \cite{houska2016augmented,Du2025ACC}, which incorporates Sequential Quadratic Programming (SQP) principles into ADMM, guaranteed global convergence for convex problems and local convergence for nonconvex problems. Its mixed-integer extensions \cite{murray2018mixed, murray2018hierarchical, murray2020convergence} provided theoretical convergence guarantees for both convex and potentially nonconvex MIPs, yet their practical implementation relies on centralized MIP solvers, which can be computationally demanding.

{\color{black}{To the best of our knowledge, the existing literature survey indicates that there is currently no reported work capable of solving distributed mixed-integer optimization problems with convergence guarantees without relying on mixed-integer programming solvers.}}
This motivates the development of a distributed solver for mixed-integer optimization that (i) operates independently of local mixed-integer solvers, (ii) provides rigorous convergence guarantees for both convex and nonconvex MINLPs.%, and (iii) achieves high computational efficiency for large-scale networked problems.

\noindent
\textbf{Main Contributions.}  Motivated by the aforementioned challenges, we propose a novel mixed-integer optimization algorithm based on Consensus ALADIN, termed \emph{Mix-CALADIN}, where all integer variables are Boolean, i.e., taking values in ${0,1}$. To the best of our knowledge, this is the first work to employ Consensus ALADIN in a framework that (i) efficiently solves distributed mixed-integer optimization problems, (ii) operates without relying on local mixed-integer solvers, and (iii) requires only the mild assumption that the objective functions are Lipschitz continuous.
Our main contributions are as follows.\\
%\noindent
\textbf{A.}
We propose a novel two-stage distributed optimization algorithm for solving distributed mixed-integer programming problems. In Stage I, Consensus ALADIN is applied directly to solve a relaxed continuous version of the problem, where the integer constraints are temporarily removed. This provides a deterministic lower bound for the original distributed mixed-integer consensus problem and also supplies a high-quality initial point for Stage II. Inspired by \cite{Hall2021, hall2025lcqpow,du2023optimal}, Stage II integrates, at the Consensus ALADIN coordinator, a constraint that enforces all relaxed integer variables to lie within $[0,1]$ and adds a penalty term to the objective function that drives the variables toward Boolean values, thereby guiding the algorithm to converge.\\
\textbf{B.} We provide convergence guarantees for Mix-CALADIN. In Stage I, the algorithm inherits all convergence properties of Consensus ALADIN, namely, global linear convergence for convex $L$-smooth problems (see Theorem~\ref{the: convex}) and local linear convergence for nonconvex problems (see Theorem~\ref{the: nonconvex}). For Stage II, we show that, under the mild assumption that the objective functions are Lipschitz continuous, the algorithm converges when the augmented Lagrangian parameter $\rho$ is chosen larger than the Lipschitz constant (see Theorem~\ref{the: stage 2}). Numerical simulations corroborate these theoretical results.
\section{Problem Formulation}\label{sec: Preliminaries}
\ADDBoYu{We consider the following distributed consensus mixed-integer optimization problem with $N$ agents: }
% and provides an overview of CALADIN, an algorithm designed for large-scale distributed consensus optimization with continuous variables.

%\subsection{Problem Formulation}\label{sec: Problem}

\begin{equation}\label{eq: mixed problem} \small
    \begin{split}
        \min_{\substack{x_i,\, z\in \mathbb R^{n_c+n_d} \\ \forall i \in \{1,\dots,N\}}} &\quad \sum_{i=1}^N f_i(x_i) \\
        \text{s.t.} &\quad x_i = z, 
    \end{split}
\end{equation}
where $x_i = [y_i^\top, b_i^\top]^\top, y_i \in \mathbb{R}^{n_c}, b_i \in \{0,1\}^{n_d}$. Each agent $i$ has a local objective function $f_i: \mathbb{R}^{n_c + n_d} \rightarrow \mathbb{R}$ and a decision variable $x_i = [y_i^\top, b_i^\top]^\top$, where $y_i \in \mathbb{R}^{n_c}$ are continuous variables and $b_i \in \{0,1\}^{n_d}$ are Boolean variables. The global consensus variable $z = [z_c^\top, z_d^\top]^\top$ comprises continuous components $z_c \in \mathbb{R}^{n_c}$ and Boolean components $z_d \in \{0,1\}^{n_d}$.
%\subsection{Preliminaries of CALADIN}
\ADDBoYu{\section{Preliminaries of CALADIN}\label{sec: Preliminaries}}
\ADDBoYu{This section provides the preliminary knowledge for our work. We begin with an overview of CALADIN, an algorithm for large-scale distributed consensus optimization with continuous variables, which lays the foundation for our approach}. 

CALADIN, see \cite{Du2025ACC}, \cite{Du2025ECC} is a distributed algorithm for solving consensus optimization problems of the form \eqref{eq: mixed problem} without boolean part,
%\begin{equation}\label{eq: continuous problem} \small
%    \begin{split}
%        \min_{\substack{x_i,\, z,\in \mathbb R^n\, \\ \forall i \in \{1,\dots,N\}}} &\quad \sum_{i=1}^N f_i(x_i)  \qquad
%        \text{s.t.}\quad x_i = z,
%    \end{split}
%\end{equation}
where all local variables $x_i$ and the global consensus variable $z$ are $n$-dimensional vectors. The augmented Lagrangian is given by:
%\begin{equation}\label{eq: Lagrangian} \small
    $\mathcal{L}(x, z, \lambda) \coloneqq \sum_{i=1}^N  f_i(x_i) + \lambda_i^\top (x_i - z) + \frac{\rho}{2} \|x_i - z\|^2,$
%\end{equation}
where $x = [x_1^\top, \dots, x_N^\top]^\top$ collects the primal variables and $\lambda = [\lambda_1^\top, \dots, \lambda_N^\top]^\top$ collects the dual variables.

The CALADIN algorithm proceeds as follows:
\begin{equation}\label{eq: CALADIN}\small
\left\{
    \begin{aligned}
        &x_i^{[k+1]} = \arg\min_{x_i}  f_i(x_i) + (\lambda_i^{[k]})^\top x_i \\
        &\qquad\qquad\qquad\qquad + \frac{\rho}{2} \|x_i - z^{[k]}\|^2 , \quad \forall i \in \mathcal{V}, \\[2mm]
        &g_i = \nabla f_i(x_i^{[k+1]}), \quad H_i \approx \nabla^2 f_i(x_i^{[k+1]}) \succ 0, \\[2mm]
        &(z^{[k]}, \lambda^{[k]}) = \left\{
            \begin{aligned}
                \min_{\Delta x_i \in \mathbb{R}^n} &\quad \sum_{i=1}^N  \Delta x_i^\top H_i \Delta x_i + g_i^\top \Delta x_i  \\
                \text{s.t.} &\quad x_i^{[k+1]} + \Delta x_i = z \mid \lambda_i, \quad \forall i.
            \end{aligned}
        \right.
    \end{aligned}
\right.
\end{equation}

The algorithm consists of three main steps: (1) each agent minimizes its augmented Lagrangian with respect to local variables; (2) all agents compute local gradients $g_i$ and positive definite Hessian approximations $H_i$; (3) a coordinator solves a consensus quadratic program (QP) using the collected local information to update the global variable $z$ and dual variables $\lambda_i$.

For continuous variables, our previous work \cite{Du2023_Arxiv,Du2025ACC,du2025decentralized}) have established local convergence guarantees for nonconvex problems and global convergence for convex problems. \ADDBoYu{The aim of this paper is to extend the CALADIN framework to handle mixed-integer optimization problems as formulated in \eqref{eq: mixed problem}.}

%\newpage
\section{Mix-CALADIN}\label{sec: algorithm}

\ADDBoYu{This section presents a novel distributed algorithm for solving problem \eqref{eq: mixed problem}. }%We begin by stating the following assumptions that are fundamental to our theoretical development, each applicable to different scenarios.

%Assumptions \ref{ass: Lip} and \ref{ass: convex} ensure global convergence of Stage I in Algorithm \ref{alg: MIX-ALADIN first}, as \ADDBoYu{we establish in Theorem \ref{the: convex} below. Assumption \ref{ass: Lip} establishes local convergence properties (see Theorem \ref{the: stage 2} below), and Assumption \ref{ass: nonconvex} guarantees local convergence of the algorithm (see Theorem \ref{the: stage 2} below).}

\subsection{Algorithm Development}\label{sec: algorithm structure}
In this section, we present our proposed two-stage distributed mixed-integer 
algorithm, detailed below as Algorithm \ref{alg: MIX-ALADIN first}. 

{\color{black} In Stage I, we consider a continuous relaxation of problem \eqref{eq: mixed problem}, constructed by relaxing the Boolean constraint. Crucially, this relaxed problem yields a theoretically guaranteed lower bound for the original problem \eqref{eq: mixed problem}, and this guarantee holds irrespective of  the convexity of the functions $f_i$. The implementation of Stage I follows the established method described in \eqref{eq: CALADIN} and is therefore not repeated here. Notably, the coordination step in \eqref{eq: coordinator} admits a closed-form solution, which reduces the computational burden compared to the original CALADIN formulation.}

% Phase I of Mix-CALADIN is minimizing a relax problem:
% \begin{equation}
%         \begin{split}
%             \mathop{\min}_{x_i,z} &\quad\sum_{i=1}^N f_i(x_i)\\
%             \text{s.t.} &\quad x_i = z  \\
%             &\quad x_i = [y_i^\top, b_i^\top]^\top,\\
%           %  &\quad y_i\in \mathbb R^{n_c},\; 0\leq b_i\leq 1.
%         \end{split}
% 		\end{equation}

\begin{algorithm}[ht]
\small
	\caption{Mix-CALADIN}
	%\textbf{Initialization:} Initial guess $(p^0,\lambda^0,\mu^0)$, choose $\Sigma_i,\rho^0,\mu^0,\epsilon$. \\
	\textbf{Initialization.} Randomly choose dual variable $\lambda_i \in \mathbb{R}^{n_c+n_d}$ for each node $i \in \mathcal{V}$, and global variable $z \in \mathbb{R}^{n_c+n_d}$. Set $\beta>1$. 
    \\
    \textbf{Stage I Iteration.}
	\begin{enumerate}
		\item[1.] Paralleled solve local NLP without subsystem coupling:
		\begin{equation}\label{eq: local OPT}
        \begin{aligned}
            x_i^{[k+1]}=&\mathop{\arg\min}_{x_i} f_i(x_i)+ (\lambda_i^{[k]})^\top x_i +\frac{\rho_1}{2} \left\|x_i-z^{[k]} \right\|^2 \\
           % &\text{s.t.} \quad 0\leq b_i \leq 1.
        \end{aligned}
			\end{equation}

\item[2.] Evaluate the Hessian and gradient:
\begin{equation}\label{eq: gradient and Hessian}
    g_i = \nabla f_i(x_i^{[k+1]}),\quad H_i\approx\nabla^2 f_i(x_i^{[k+1]})\succ 0.
\end{equation}
		
		\item[3.] The coordination:
		\begin{equation}\label{eq: coordinator}\left\{
        \begin{aligned}
&z^{[k+1]}=\left( \sum_{i=1}^N H_i\right)^{-1}\left(\sum_{i=1}^N H_i x_i^{[k+1]} -g_i \right)\\
& \lambda_i^{[k+1]} = H_i(x_i^{[k+1]} - z^{[k+1]} ) -g_i. 
          %  &\quad 0\leq z_{d}\leq 1.
        \end{aligned}\right.
		\end{equation}
	\end{enumerate}
   % Projection operation: $ z^{[k+1]} = \Pi_{[0,1]} (\tilde z^{[k+1]})$.
    If $\|z^{[k+1]}-z^{[k]}\|\leq \epsilon,$ then switch to Stage II with $\tilde z^* = z^{[k+1]}$ \\[1mm]
    %\\\hrulefill
    \textbf{Stage II}\\
    \textbf{Initialization.} Set $z^1 = \tilde z^*$ from Stage I. Set $\alpha = 1, k = 1$.\\
    \textbf{Iteration.}
     \begin{enumerate}
  \item[1.]  Inner level: Do iteration \emph{k}:\\
    % \item With CasaDi\begin{equation}
    %     y_i^{[k+1]} = \mathop{\arg\min}_{y_i} f_i(y_i, z_d) + \lambda_i^\top [y_i^\top, z_d^\top]^\top + \frac{\sigma}{2} \| y_i - z_c \|^2
    % \end{equation}
    % and transmit it to the coordinator.
    \begin{itemize}
        \item For each agent: evaluate the gradient and transmit it to the coordinator:
\begin{equation}\label{eq: GRADIENT EVA}
    g_i = \nabla f_i(z^{[k]}).%,\quad H_i\approx\nabla^2 f_i(x_i^+).
\end{equation}		
		\item For coordinator: receive $g_i$ from each agent      
  %       		\begin{equation}\label{eq: stage 2 QP}
  %       \begin{split}
  %           z^{[k+1]}=\mathop{\arg\min}_{z,\Delta X_i} &\quad \frac{\sigma}{2} \sum_{i=1}^N \|\Delta x_i\|^2 + \sum_{i=1}^N g_i^\top z   \\
  %           &+ \alpha  (\mathbf{1}- 2z_{d}^{[k]})^\top z_{d}\\%  ^\top H_i \Delta x_i \\
  %           \text{s.t.} &\quad x_i^+ + \Delta x_i = z | \lambda_i\\
  %           &\quad 0\leq z_{d}\leq 1.\\
  %          % &\quad (\mathbf{1}- 2z_{d}^-)^\top z_{d} \leq \gamma
  %       \end{split}
		% \end{equation} 
and transmit $z^{[k]}$ to each agent.\\
		\begin{equation}\label{eq: stage 2 QP}
        \begin{split}
            z^{[k+1]}=\mathop{\arg\min}_{z} &\quad \frac{N\rho_2}{2} \|z-z^{[k]}\|^2 + z^\top\sum_{i=1}^N g_i  \\
            &+ \alpha  (\mathbf{1}- 2z_{d}^{[k]})^\top z_{d}\\%  ^\top H_i \Delta x_i \\
            \text{s.t.} %&\quad x_i^+ + \Delta x_i = z | \lambda_i.\\
            &\quad 0\leq z_{d}\leq 1.\\
           % &\quad (\mathbf{1}- 2z_{d}^-)^\top z_{d} \leq \gamma
        \end{split}
		\end{equation} 
    \end{itemize}
      If $\|z^{[k+1]}-z^{[k]}\|\leq \epsilon_{inner}$, then output to the \emph{outer level} $z^* = z^{[k+1]}$. Else set \emph{k=k+1} and go to the \emph{inner level}.

    \item[2.]  Outer level: If $(1-z_d^*)^\top z_d^* \geq \epsilon_{outer}$, then the coordinator sets $\alpha = \beta\alpha, k=k+1, z^{[k]} = z^*$ %and $\rho = \beta \rho$, 
and go to the \emph{inner level}. Else terminate operation.\\
\item[3.] Output: $z^*$.
    \end{enumerate}
	\label{alg: MIX-ALADIN first}
\end{algorithm}

Stage II of Algorithm \ref{alg: MIX-ALADIN first} is initialized using the solution from Stage I by setting $z^{[1]} = \tilde {z}^*$ and implements a bi-level optimization framework. The inner level is inspired by \cite[Algorithm~3)]{wu2025time}, where information exchange is managed solely by a coordinator while each agent performs only gradient evaluation. Specifically, at each iteration, the inner level executes two steps: (i) each agent computes its local gradient $g_i$ at $z^{[k]}$ according to \eqref{eq: GRADIENT EVA}, and (ii) the coordinator solves the convex quadratic program in \eqref{eq: stage 2 QP} with box constraint $0 \leq z_d \leq 1$, employing the term $\alpha (\mathbf{1} - 2z_d^{[k]})^\top z_d$ to drive the continuous variable $z_d$ toward Boolean values as $\alpha$ increases, a strategy motivated by \cite{zhu2022alternating,du2023optimal}. The inner level iterates until $\|z^{[k+1]} - z^{[k]}\| \leq \epsilon_{\text{inner}}$, outputting $z^* = z^{[k+1]}$ to the outer level or proceeding to iteration $k+1$ otherwise. In the outer level, if $(\mathbf{1} - z_d^{[k+1]})^\top z_d^{[k+1]} \geq \epsilon_{\text{outer}}$, the penalty parameter is increased via $\alpha \leftarrow \beta\alpha$ ($\beta > 1$) and the inner level is restarted with $k \leftarrow k + 1$ and $z^{[k]} = z^*$; otherwise, $z^*$ is returned as the final solution. The theoretical convergence analysis is provided in the subsequent section.   Note that, Stage I of Algorithm~\ref{alg: MIX-ALADIN first} (i.e., CALADIN) could be replaced by other distributed optimization algorithms to initialize Stage II. However, since CALADIN offers favorable convergence guarantees and numerical performance for both convex and non-convex problems \cite{Du2025ACC}, this paper focuses on CALADIN.

\begin{remark} \label{the: nonconvex ensure positive definiteness}\textbf{\emph{Lightweight Implementation:}}
    For nonconvex problems, the Hessian matrices $\nabla^2 f_i$ are not guaranteed to be positive definite. To ensure positive definiteness in Stage I of Algorithm \ref{alg: MIX-ALADIN first}, we employ the following regularization scheme:
    \begin{equation}\small
        H_i = \begin{cases}
            \nabla^2 f_i(x_i^{[k+1]}), & \text{if } \sigma_i > 0, \\
            \nabla^2 f_i(x_i^{[k+1]}) + 1.1(|\sigma_i| + 0.1) \cdot I, & \text{if } \sigma_i  \leq 0,
        \end{cases}
    \end{equation}
    where $\sigma_i$ denotes the minimum eigenvalue of $\nabla^2 f_i(x_i^{[k+1]})$.

    For convex problems, we simplify the implementation by replacing \eqref{eq: coordinator} with:
    \begin{equation}\label{eq: coordinator2}\small\left\{
        \begin{aligned}
            &z^{[k+1]} = \frac{1}{N} \left( \sum_{i=1}^N x_i^{[k+1]} - \frac{g_i}{\rho} \right), \\
            &\lambda_i^{[k+1]} = \rho(x_i^{[k+1]} - z^{[k+1]}) - g_i.
        \end{aligned}\right.
    \end{equation}
    See \cite{Du2025ACC} for details.
\end{remark}

\begin{remark} \label{the: nonconvex penalty term}\textbf{\emph{Smoothness Approximation for Boolean Variables:}}
    To enforce convergence of $z_d$ to Boolean solutions, one could incorporate the nonconvex term
    \begin{equation}\label{eq: unstable}
        \alpha (\mathbf{1} - z_d)^\top z_d
    \end{equation}
    into \eqref{eq: stage 2 QP} with sufficiently large $\alpha$. However, this approach may lead to numerical instability. Inspired by \cite{Hall2021a,Hall2021,hall2025lcqpow,zhu2022alternating,du2023optimal}, we instead use the convex approximation $\alpha (\mathbf{1} - 2z_d^{[k]})^\top z_d$ at each iterate $z_d^{[k]}$ to maintain numerical robustness while achieving the desired convergence behavior. Note that the convergence of the centralized linear complementarity QP (LCQP) discussed in \cite{Hall2021} relies on step-size adjustment for the primal variable, i.e., a globalization strategy. In contrast, due to the inherent distributed optimization nature of ALADIN, no such globalization strategy is required. Moreover, aside from \cite{le2025distributed} (numerically unstable) and \cite{takapoui2017alternating} (heuristic, no convergence guarantee), no existing distributed optimization algorithm operates without a centralized mixed-integer programming solver.
\end{remark}

% \begin{algorithm}[ht]
% 	\caption{Mix-CALADIN: the second stage Discrete Approximation}
% 	%\textbf{Initialization:} Initial guess $(p^0,\lambda^0,\mu^0)$, choose $\Sigma_i,\rho^0,\mu^0,\epsilon$. \\
% 	\textbf{Repeat:}
% 	\begin{enumerate}
% 		% \item Paralleled solve local NLP without subsystem coupling:
% 		% \begin{equation}\label{eq: local step}
%   %       \begin{aligned}
%   %           x_i^+=&\mathop{\arg\min}_{x_i} f_i(x_i)+ \lambda_i^\top x_i +\frac{\rho}{2} \left\|x_i-z \right\|^2 \\
%   %           &\text{s.t.} \quad 0\leq b_i \leq 1.
%   %       \end{aligned}
% 		% 	\end{equation}

% \item Evaluate the Hessian and gradient:
% \begin{equation}
%     g_i = \nabla f_i(z).%,\quad H_i\approx\nabla^2 f_i(x_i^+).
% \end{equation}
		
% 		\item The coordination:
% 		\begin{equation}\label{eq: stage 2 QP}
%         \begin{split}
%             z=\mathop{\arg\min}_{z} &\quad \frac{N\rho}{2} \|z-z^{[k]}\|^2 + z^\top\sum_{i=1}^N g_i  \\
%             &+ \alpha  (\mathbf{1}- 2z_{d}^{[k]})^\top z_{d}\\%  ^\top H_i \Delta x_i \\
%             \text{s.t.} %&\quad x_i^+ + \Delta x_i = z | \lambda_i.\\
%             &\quad 0\leq z_{d}\leq 1.\\
%            % &\quad (\mathbf{1}- 2z_{d}^-)^\top z_{d} \leq \gamma
%         \end{split}
% 		\end{equation}
% 	\end{enumerate}
% 	\label{alg: MIX-ALADIN second}
% \end{algorithm}

\subsection{Convergence Analysis}\label{sec: convergence}
This section presents the convergence analysis of Algorithm \ref{alg: MIX-ALADIN first}, structured around three main theorems. Theorems \ref{the: convex} and \ref{the: nonconvex} establish the convergence properties of Stage I \ADDBoYu{for both convex and nonconvex scenarios}, while Theorem \ref{the: stage 2} addresses the convergence of Stage II. \ADDBoYu{Let us note that the specific assumptions required for each theorem are mentioned in the corresponding theorem.} Prior to presenting Theorem \ref{the: stage 2}, we introduce a supporting lemma that facilitates the proof.

\begin{assumption}\label{ass: Lip}
    For each agent \ADDBoYu{$i$}, the local cost function $f_i: \mathbb{R}^n \rightarrow \mathbb{R}$ is closed, proper, and $L$-smooth, i.e., there exists a Lipschitz constant $L > 0$ such that for all $x_a, x_b \in \mathbb{R}^n$:
    \begin{equation}\label{eq: lip}\small
        f_i(x_a) \leq f_i(x_b) + \nabla f_i(x_b)^\top (x_a - x_b) + \frac{L}{2}\|x_a - x_b\|^2.
    \end{equation}
\end{assumption}

\begin{assumption}\label{ass: convex}
    For each agent \ADDBoYu{$i$}, the local cost function $f_i: \mathbb{R}^n \rightarrow \mathbb{R}$ is closed, proper, and strongly convex, i.e., there exists a  constant $\mu_i > 0$ such that for all $x_a, x_b \in \mathbb{R}^n$:
    \begin{equation}\small\label{eq: mu}\small
\begin{split}
    f_i(x_\alpha)+&\nabla f_i(x_\alpha)^\top (x_\beta-x_\alpha) +\frac{\mu_i}{2}\|x_\beta-x_\alpha\|^2\leq f_i(x_\beta). 
\end{split}
\end{equation}
\end{assumption}

\begin{assumption}\label{ass: nonconvex}
    For each agent \ADDBoYu{$i$}, the local cost function $f_i: \mathbb{R}^n \rightarrow \mathbb{R}$ is closed, proper, twice continuously differentiable, and satisfies the second-order sufficient condition at a local minimizer $x^*$ of \eqref{eq: mixed problem} without boolean part, i.e., $\nabla^2 f_i(x^*) \succ 0$.
\end{assumption}

\begin{theorem}\label{the: convex}\textbf{\emph{Global Linear Convergence  for Convex Problems:}}
Suppose each agent $i$ has a local cost function $f_i$ satisfying Assumptions \ref{ass: Lip} and \ref{ass: convex}. Then, for any $\rho_1 > 0$, Stage I of Algorithm \ref{alg: MIX-ALADIN first} exhibits global linear convergence.
\end{theorem}
\textbf{Proof.} \ADDBoYu{The result in \cite[Theorem 2]{Du2025ACC} establishes global linear convergence of C-ALADIN for smooth and strongly convex problems. The global linear convergence under Assumptions \ref{ass: Lip} and \ref{ass: convex} follows a similar analysis as \cite[Theorem 2]{Du2025ACC}. It is omitted here due to space limitations.}  \hfill $\blacksquare$
%establishes global linear convergence of C-ALADIN for smooth and strongly convex problems, and since Theorem \ref{the: convex} employs stronger assumptions (Assumptions \ref{ass: Lip} and \ref{ass: convex}), the convergence rate under these enhanced conditions is guaranteed to be at least globally linear. \hfill $\blacksquare$

\begin{theorem}\label{the: nonconvex}\textbf{\emph{Local Linear Convergence for nonconvex Problems:}}
Suppose each agent $i$ has a local cost function $f_i$ satisfying Assumption \ref{ass: nonconvex}. Then, for a sufficiently large $\rho_1 > 0$, Stage I of Algorithm \ref{alg: MIX-ALADIN first} exhibits local linear convergence.
\end{theorem}
\textbf{Proof.} See \cite[Theorem 3]{Du2025ACC}. \hfill $\blacksquare$

We now establish the convergence theory for Stage II of Algorithm \ref{alg: MIX-ALADIN first}. First, we present a key lemma that will support the subsequent proof of Theorem \ref{the: stage 2}.

\begin{lemma}\label{lemma}
Suppose each agent $i$ has a local cost function $f_i$ satisfying Assumption \ref{ass: Lip}. Define $\gamma^{[k]} \coloneqq (\mathbf{1} - z_d^{[k]})^\top z_d^{[k]}$. Then the following inequality holds:
\begin{equation}\label{eq: lemma}\small
    \begin{split}
        &\sum_{i=1}^N g_i(z^{[k]})^\top (z^{[k+1]} - z^{[k]}) \\
        \leq &\alpha (\gamma^{[k]} - \gamma^{[k+1]}) - \frac{N\rho_2}{2} \|z^{[k+1]} - z^{[k]}\|^2.
    \end{split}
\end{equation}
\end{lemma}

\textbf{Proof.} See Appendix \ref{Appenmdix Lemma1}.  \hfill $\blacksquare$

Theorem \ref{the: stage 2} establishes the convergence of the inner level of Stage II in Algorithm \ref{alg: MIX-ALADIN first} by demonstrating the monotonic decrease of the energy function:
\begin{equation}\label{eq: energy}\small
    E(z^{[k]}) = \sum_{i=1}^N f_i(z^{[k]}) + \alpha \gamma^{[k]}.
\end{equation}

\begin{theorem}\label{the: stage 2}\textbf{\emph{Local Convergence of Stage II of Algorithm \ref{alg: MIX-ALADIN first}:}}
Suppose each agent $i$ has a local cost function $f_i$ satisfying Assumption \ref{ass: Lip}. For fixed $\rho_2 > L$, the energy function satisfies:
\begin{equation}\label{eq: theorem 3}\small
    E(z^{[k+1]}) < E(z^{[k]}).
\end{equation}
\end{theorem}

\textbf{Proof.} See Appendix \ref{Appendix Theorem Stage 2}.  \hfill $\blacksquare$

\ADDBoYou{With the inner-loop convergence of Stage II secured by Theorem \ref{the: stage 2}, the role of the outer loop is to increase $\alpha$, a strategy that drives the entire algorithm toward a solution satisfying the Boolean constraint. Specifically, this increase is triggered once the change in the inner-loop variable z is below $\epsilon_{inner}$, a criterion which ensures that the inner-loop has sufficiently converged before tightening the constraint.}
%Since the inner level of Stage II in Algorithm \ref{alg: MIX-ALADIN first} guarantees convergence (Theorem  \ref{the: stage 2}, the role of the outer loop ), we now consider increasing $\alpha$ once the change in the inner level variable $z$ reaches the threshold $\epsilon_{inner}$. 
Following \cite[Theorem 1]{Hall2021} and \cite[Section 5]{ralph2004some}, it can be shown that as $\alpha$ gradually increases and becomes sufficiently large, the Boolean constraint on the variables will be satisfied and the outer level of Stage II will converge to a local optimum of problem \eqref{eq: mixed problem}.

\begin{remark} \label{the: reformulation to accelerate convergence rate} \textbf{\emph{Enhance Accelerated Convergence for Stage II:}}
To enhance the numerical convergence performance, analogous to the analysis in Theorem \ref{the: stage 2}, we consider the following reformulation of the quadratic program in \eqref{eq: stage 2 QP}, under the assumption that $(H_i + \rho_2 I) \succ L I$:

\begin{equation}\label{eq: stage 2 QP2}\small
    \begin{split}
        z^{[k+1]} = \arg\min_{z} \quad & \sum_{i=1}^N \frac{1}{2} (z - z^{[k]})^\top (H_i + \rho_2 I) (z - z^{[k]}) \\
        & + z^\top \sum_{i=1}^N g_i + \alpha (\mathbf{1} - 2z_{d}^{[k]})^\top z_{d} \\
        \text{s.t.} \quad & 0 \leq z_{d} \leq 1.
    \end{split}
\end{equation}

The convergence analysis for this reformulation follows that of Theorem \ref{the: stage 2} and is therefore omitted. This reformulation is designed to improve the convergence rate of the algorithm. \ADDBoYu{The acceleration achieved stems from the explicit use of second-order information, which provides a more accurate local model of the curvature, thus leading to faster convergence.}
\end{remark}

\section{Numerical Simulation and Results}\label{sec: numerical}
To validate the performance of our Mix-CALADIN algorithm for mixed-integer programs, numerical experiments are carried out in MATLAB R2025a on a PC equipped with an AMD Ryzen7-8745HS processor and 16GB of RAM.
{\color{black}{Based on the sensor localization problem in \cite{Du2025ACC}, we adapt the problem formulations to distributed consensus mixed-integer settings and conduct extensive numerical simulations on two representative cases to evaluate the performance of the proposed algorithm.}}

The first case studies a nonconvex sensor localization problem with mixed-Boolean constraints, designed to demonstrate the algorithm's capability in handling challenging nonconvex optimization landscapes. The problem is formally formulated as follows:
\begin{equation}\label{eq: mixed problem3}\small
    \begin{split}
        \min_{\substack{x_i,\, z,\, \\ \forall i \in \{1,\dots,N\}}} 
        &\quad \sum_{i=1}^N \bigg( \frac{1}{2} \|{y}_i - {\zeta}_i^{\alpha}\|^2 + \frac{1}{2}\|{b}_i - {\zeta}_i^{\beta}\|^2 \\
        &\quad + \frac{1}{2}\sum_{j=1}^{n_c}\left(\left({y}_i[j] - {b}_i[j]\right)^2 - {\zeta}_i^{\gamma}[j]\right)^2 \bigg) \\
        \text{s.t.}\hspace{0.5cm} &\quad x_i = z, \\
        &\quad x_i = [{y}_i^\top, {b}_i^\top]^\top, \\
        &\quad {y}_i \in \mathbb{R}^{n_c}, \quad {b}_i \in \{0,1\}^{n_d}.
    \end{split}
\end{equation}

The number of agents is $N=20$, the dimensions of the continuous and boolean parts of the local variable $x_i$ are $n_c=10$ and $n_d=10$, respectively. The parameter vectors ${\zeta}_i^{\alpha} \in \mathbb{R}^{n_{c}}$, ${\zeta}_i^{\beta} \in \mathbb{R}^{n_{d}}$, ${\zeta}_i^{\gamma} \in \mathbb{R}^{n_{d}}$ are drawn randomly from a standard normal distribution. Note that the element-wise operation $({y}_i - {b}_i)^2$ in the objective function requires $n_c = n_d$. Thus, the dimension of ${\zeta}_i^{\gamma}$ is $n_d$ (or $n_c$).

The second case studies a convex objective function, allowing for a performance benchmark against established distributed optimization methods, specifically the projection-based ADMM technique \cite{takapoui2020simple}. This approach does not rely on mixed-integer programming solvers; instead, it first employs a standard continuous variable solution method and then applies rounding. The parameter set ($N, n_c, n_d, \zeta_i^{\alpha}, \zeta_i^{\beta}$) is shared with the first problem, with the removal of the nonconvex term being the only difference. This problem takes the form: 
\begin{equation}\label{eq: mixed problem2}\small
    \begin{split}
        \min_{\substack{x_i,\, z,\, \\ \forall i \in \{1,\dots,N\}}} 
        &\quad \sum_{i=1}^N \bigg( \frac{1}{2}\|{y}_i - {\zeta}_i^{\alpha}\|^2 + \frac{1}{2}\|{b}_i - {\zeta}_i^{\beta}\|^2 \bigg) \\
        \text{s.t.}\hspace{0.5cm} &\quad x_i = z, \\
        &\quad x_i = [{y}_i^\top, {b}_i^\top]^\top, \\
        &\quad {y}_i \in \mathbb{R}^{n_c}, \quad {b}_i \in \{0,1\}^{n_d}.
    \end{split}
\end{equation}

For a fair comparison, the parameters of all algorithms are meticulously tuned. The penalty parameter $\rho_1$ for the augmented Lagrangian function \eqref{eq: local OPT}, common to both Stage I of Algorithm \ref{alg: MIX-ALADIN first} and projection-based ADMM, is set to $10^5$ for nonconvex problem and $10$ for convex problem. The penalty parameter $\rho_2$ for the QP function \eqref{eq: stage 2 QP} is set to $10^5$ for nonconvex problem and $10$ for convex problem.

Fig.~\ref{fig: stage1} illustrates the convergence behavior of the global variable $z$ in Stage I of Algorithm \ref{alg: MIX-ALADIN first} applied to both nonconvex and convex optimization problems. The vertical axis shows the residual norm $\|z^{[k]} - \tilde z^{*}\|^2$ on a logarithmic scale, while the horizontal axis tracks the iteration count up to $200$. Both curves exhibit linear convergence rates, which aligns well with the theoretical findings presented in Theorem 1 and Theorem 2. 
\ADDBoYou{It is worth noting that the two curves do not start from the same point since the underlying problems are fundamentally different, with one being convex and the other nonconvex, consequently leading to distinct optimal solutions $\tilde{z}^*$}. 
Stage I facilitates effective convergence across both convex and nonconvex problems of the form \eqref{eq: mixed problem} by relaxing the mixed-integer variables into a continuous space. This performance aligns with the theoretical expectations established in Theorems \ref{the: convex} and \ref{the: nonconvex}, respectively. These results provide a reference point via relaxation, establishing a lower bound for the subsequent Stage II of our algorithm.

%{\color{blue}{In Stage I, the algorithm effectively solves the relaxed optimization problem \eqref{eq: mixed problem}, in which the mixed-integer constraints have been relaxed. Convergence is achieved for both convex and nonconvex formulations: the convex case exhibits rapid convergence, whereas the nonconvex case converges more slowly, which aligns with the theoretical expectations for these two problem classes, as supported by Theorem \ref{the: convex} and Theorem \ref{the: nonconvex}. These results, obtained via constraint relaxation, provide a reference lower bound for the subsequent Stage II of the algorithm.}}
\begin{figure}[ht]
	\centering
\includegraphics[width=0.4\textwidth,height=0.2\textheight]{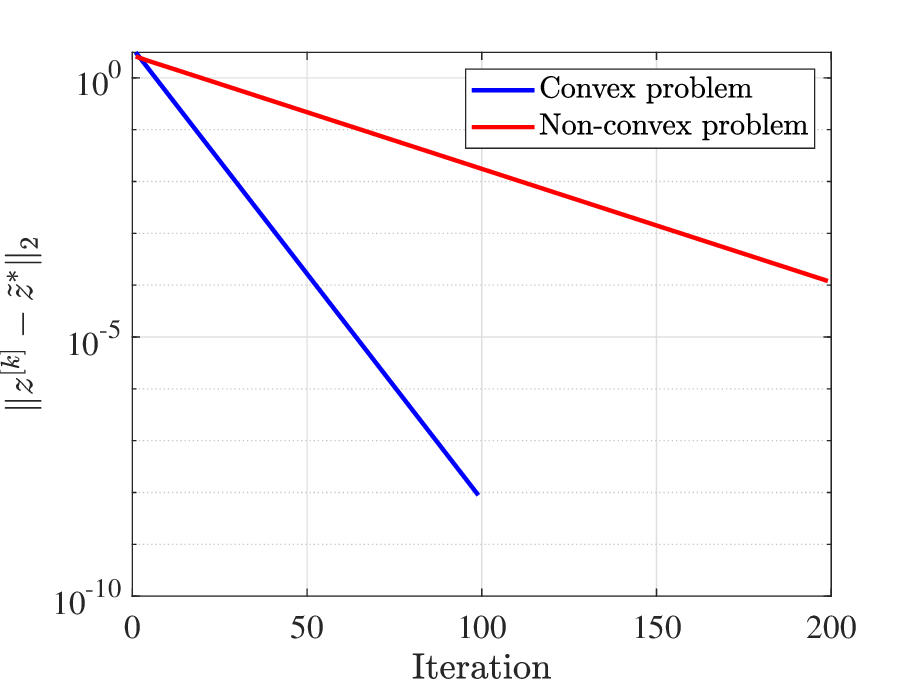}
	\caption{A comparative analysis of Stage I performance in Algorithm \ref{alg: MIX-ALADIN first} is conducted for both convex and nonconvex problems, with parameter settings $\rho_1 = 10$ for convex cases \eqref{eq: mixed problem2} and $\rho_1 = 10^5$ for nonconvex cases \eqref{eq: mixed problem3}.}
	\label{fig: stage1}
\end{figure}

Fig.~\ref{fig: stage2} illustrates the convergence behavior of the global variable $z$ in Stage II of Algorithm \ref{alg: MIX-ALADIN first} for nonconvex and convex optimization problems. The results demonstrate that the algorithm effectively converges while the residual is toward zero in both cases, with the convex problem converging after $199$ iterations and the nonconvex problem after $237$ iterations. 

\begin{figure}[ht]
	\centering
\includegraphics[width=0.4\textwidth,height=0.2\textheight]{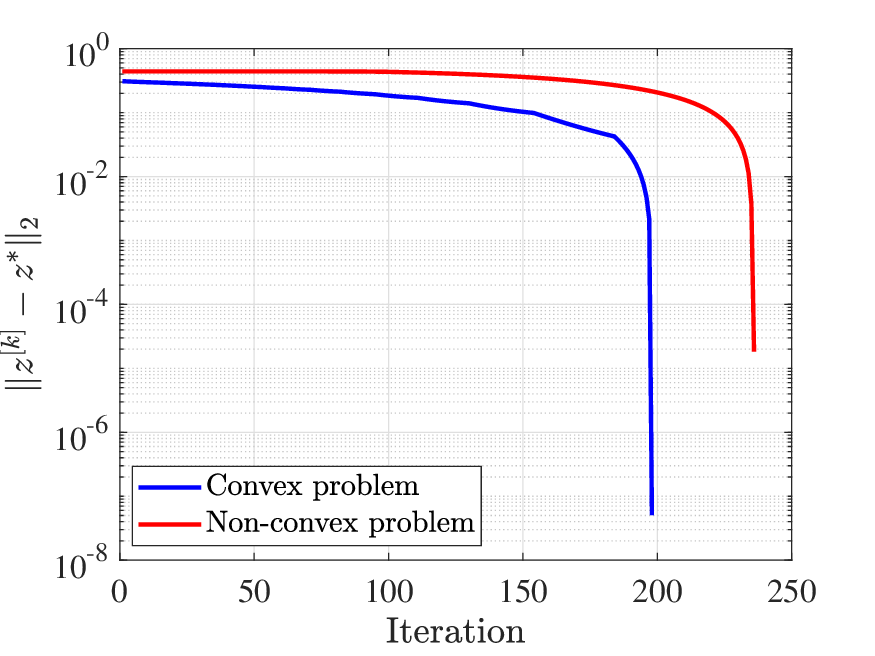}
	\caption{A comparative analysis of Stage II performance in Algorithm \ref{alg: MIX-ALADIN first} is conducted for both convex and nonconvex problems, with parameter settings $\rho_2 = 10$ for the convex case \eqref{eq: mixed problem2} and $\rho_2 = 10^5$ for the nonconvex cases \eqref{eq: mixed problem3}.}
	\label{fig: stage2}
\end{figure}

Fig.~\ref{fig: Convex energy function value} and Fig.~\ref{fig: Nonconvex energy function value} present the energy function \eqref{eq: energy} profiles throughout the optimization process for the convex and nonconvex problems, respectively. \ADDBoYu{In Fig.~\ref{fig: Convex energy function value}}, the distinctive staircases pattern is observed in the convex case, characterized by $20$ abrupt increases followed by gradual declines. This behavior is attributed to the outer-loop updates of the parameter $\alpha$: each increase in $\alpha$ temporarily elevates the energy function value, after which the inner-loop optimization reduces it toward optimality. This process terminates in satisfaction of the Boolean constraint. \ADDBoYu{In  Fig.~\ref{fig: Nonconvex energy function value},} the nonconvex problem shows more frequent rises in the energy function value within the first $102$ iterations, indicating frequent exits from the inner-loop due to insufficient $\alpha$ values. Once $\alpha$ becomes sufficiently large after $102$ iterations, the energy function decreases monotonically, converging near its initial value by iteration $237$ and satisfying the Boolean constraint.

\begin{figure}[ht]
	\centering
\includegraphics[width=0.4\textwidth,height=0.2\textheight]{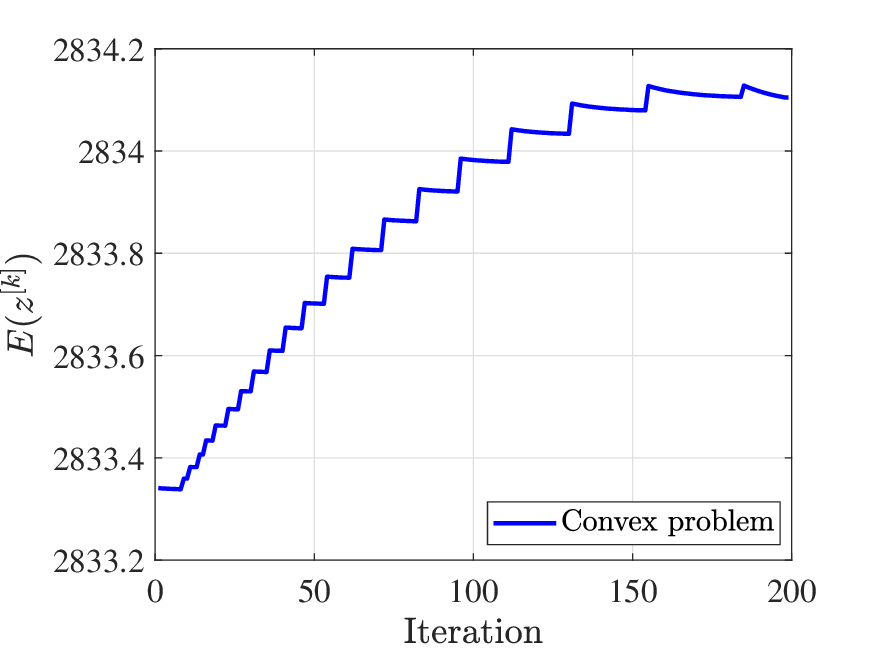}
	\caption{Evolution of the energy function $E(z^{[k]})$ (see \eqref{eq: energy}) during Stage II of Algorithm \ref{alg: MIX-ALADIN first} for the convex problem \eqref{eq: mixed problem2}.}
	\label{fig: Convex energy function value}
\end{figure}

\begin{figure}[ht]
	\centering
\includegraphics[width=0.4\textwidth,height=0.2\textheight]{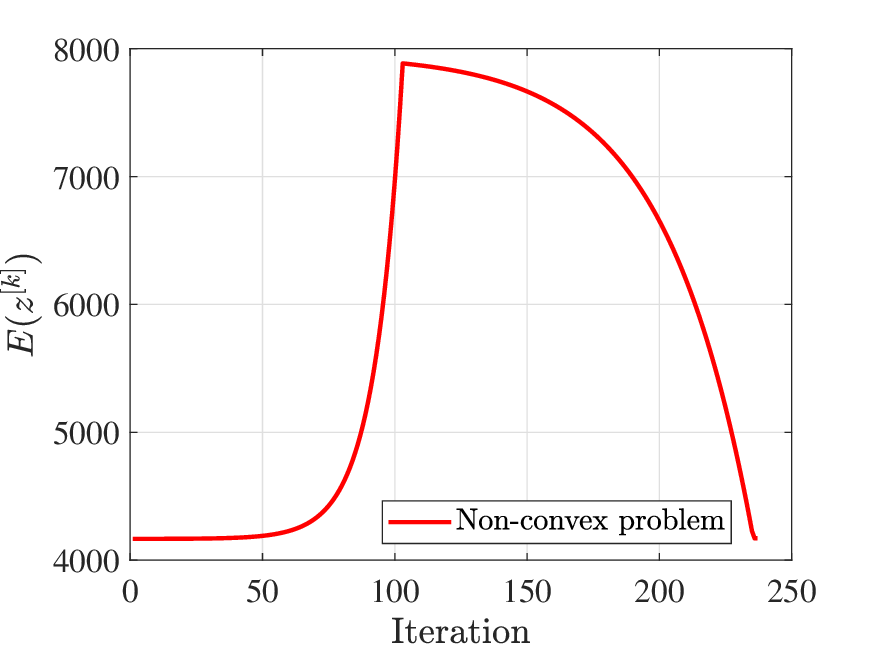}
	\caption{Evolution of the energy function $E(z^{[k]})$  during Stage II of Algorithm \ref{alg: MIX-ALADIN first} for the nonconvex problem \eqref{eq: mixed problem3}.}
	\label{fig: Nonconvex energy function value}
\end{figure}

Figure~\ref{fig: Comparison} shows the comparative performance between the projection-based ADMM \cite{takapoui2020simple} and Algorithm \ref{alg: MIX-ALADIN first} in a convex optimization problem. It is important to note that ADMM is not generally applicable to problems containing nonconvex terms. Furthermore, the projection-based method serves as a heuristic baseline and lacks convergence theory guarantees. As shown by the blue curve, the projection-based ADMM method converges to a local minimum. In contrast, Algorithm \ref{alg: MIX-ALADIN first} (red curve) proceeds in two distinct \ADDBoYu{stages}: in Stage I, it rapidly converges to the relaxed solution of the problem with continuous relaxation for the Boolean constraint, providing a high-quality initial point for Stage II. The introduction of exact Boolean constraint in Stage II leads to a temporary increase in the objective value, reflecting the formulation's shift towards feasibility. After 199 iterations, Algorithm \ref{alg: MIX-ALADIN first} converges to a solution that is superior to that attained by the projection-based ADMM approach, demonstrating both improved performance and more rigorous convergence behavior with theoretical guarantee.

\begin{figure}[ht]
	\centering
\includegraphics[width=0.5\textwidth,height=0.29\textheight]{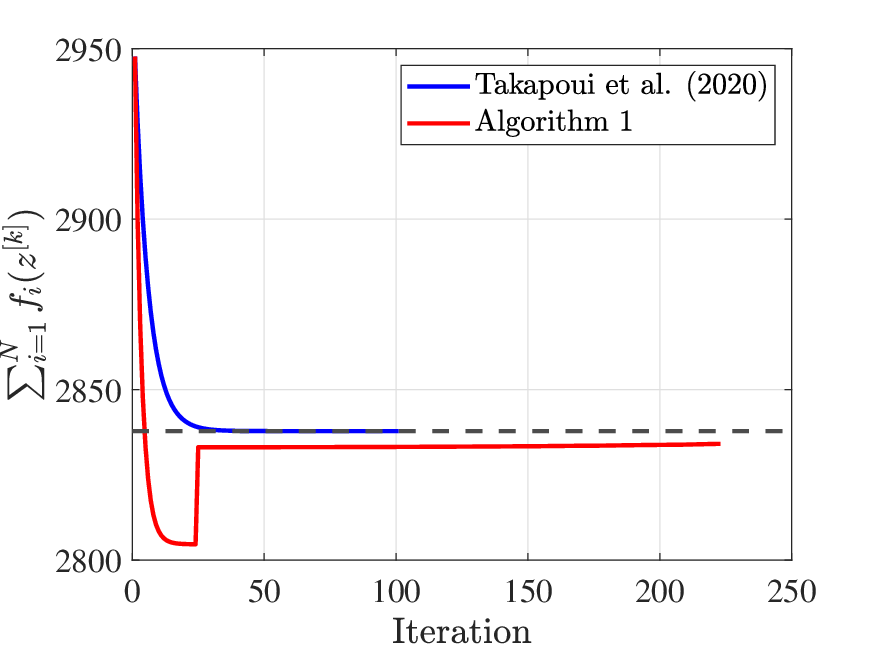}
	\caption{Convergence comparison of Algorithm \ref{alg: MIX-ALADIN first} and \cite{takapoui2020simple}.}
	\label{fig: Comparison}
\end{figure}

In summary, the simulation results consistently demonstrate that the proposed two-stage algorithm is capable of solving both mixed-integer convex and nonconvex optimization problems. The algorithm exhibits  rapid convergence, high-quality solution and strong robustness. The efficiency of the algorithm hinges on its two-stage architecture. The first stage acts as an efficient initialization step, which relaxes the Boolean constraint into continuous one and applies the well-established CALADIN algorithm to swiftly obtain a high-quality solution, serving as a superior initial point for Stage II. Then Stage II demonstrates a clear process of refinement for both convex and nonconvex problems.  
%\ADDBoYou{Its inner-loop efficiently solves the optimization subproblem for a given $\alpha$, while the outer-loop ensures strict Boolean constraint satisfaction by adaptively adjusting the parameter $\alpha$, thereby driving the algorithm toward a feasible and optimal solution, with the Boolean constraint fully satisfied upon convergence}. 
{\color{black}{The algorithm's inner loop solves the subproblem for a given $\alpha$, while the outer loop adjusts $\alpha$ to enforce Boolean constraints, ensuring convergence to a feasible solution}}. A key strength of Algorithm \ref{alg: MIX-ALADIN first} is its theoretical convergence guarantee, a foundation advantage over heuristic-based approaches like the projection-based ADMM used for comparison, which lack such theoretical underpinning. \ADDBoYu{The direct comparative results demonstrate not only the effectiveness of Algorithm \ref{alg: MIX-ALADIN first} but also its superiority in the quality of the final solution.}

\section{Conclusion}
In this paper, we presented a novel distributed algorithm for consensus mixed-integer optimization named Mix-CALADIN. Our algorithm efficiently handles both convex and nonconvex optimization problems while eliminating the dependency on conventional mixed-integer solvers. Numerical simulations validate the algorithm's performance and demonstrate its advantages over existing approaches in the literature. 

Future work will focus on extending the proposed techniques to resource-allocation-based ALADIN variants and other distributed algorithms, as well as establishing the convergence rate analysis for Stage II of Mix-CALADIN. Additionally, we plan to explore the extension of our method to open network settings, where the framework structure may change over time.

\appendix
\section{Proof of Lemma \ref{lemma}}\label{Appenmdix Lemma1}

We begin by introducing two auxiliary functions:
\begin{equation}\label{eq: aux1}\small
    \begin{split}
        &\phi^{[k]}(z) \coloneqq \frac{N\rho_2}{2} \|z - z^{[k]}\|^2 \\
        &+ \sum_{i=1}^N g_i(z^{[k]})^\top (z - z^{[k]}) + \alpha (1 - z_d)^\top z_d,
    \end{split}
\end{equation}
\begin{equation}\label{eq: aux2}\small
    \begin{split}
        &\psi^{[k]}(z) \coloneqq \frac{N\rho_2}{2} \|z - z^{[k]}\|^2 \\
        &+ \sum_{i=1}^N g_i(z^{[k]})^\top (z - z^{[k]}) + \alpha (1 - 2z_d^{[k]})^\top z_d,
    \end{split}
\end{equation}
with $0 \leq z_d \leq 1$.

Note that $\nabla \phi^{[k]}(z^{[k]})^\top (z^{[k+1]} - z^{[k]}) = \nabla \psi^{[k]}(z^{[k]})^\top (z^{[k+1]} - z^{[k]})$. Since $\psi^{[k]}$ is convex, we have $\nabla \psi^{[k]}(z^{[k]})^\top (z^{[k+1]} - z^{[k]}) \leq 0,$
which implies $\nabla \phi^{[k]}(z^{[k]})^\top (z^{[k+1]} - z^{[k]}) \leq 0$ (see also \cite[Lemma 2]{Hall2021}).

From the optimality condition of the QP in \eqref{eq: stage 2 QP}, we obtain:
\begin{equation}\label{eq: important 2}\small
    \begin{split}
       & \frac{N\rho_2}{2} \|z^{[k+1]} - z^{[k]}\|^2 \\
       &+ (z^{[k+1]} - z^{[k]})^\top \sum_{i=1}^N g_i(z^{[k]}) + \alpha \gamma^{[k+1]} \leq \alpha \gamma^{[k]}.
    \end{split}
\end{equation}
Rearranging terms yields \eqref{eq: lemma}. 

\section{Proof Theorem \ref{the: stage 2}}\label{Appendix Theorem Stage 2}
Since Assumption \ref{ass: Lip} holds for all $f_i$, we have:
\begin{equation}\label{eq: lip2}\small
    \begin{split}
        \sum_{i=1}^N f_i(z^{[k+1]}) \overset{\eqref{eq: lip}}{\leq} & \sum_{i=1}^N f_i(z^{[k]}) + g_i(z^{[k]})^\top (z^{[k+1]} - z^{[k]}) \\
        & + \frac{NL}{2} \|z^{[k+1]} - z^{[k]}\|^2.
    \end{split}
\end{equation}

Combining \eqref{eq: lip2} with Lemma \ref{lemma} yields:
\begin{equation}\label{eq: stage 2}\small
    \begin{aligned}
        \sum_{i=1}^N f_i(z^{[k+1]}) \leq & \sum_{i=1}^N f_i(z^{[k]}) + \frac{N(L - \rho_2)}{2} \|z^{[k+1]} - z^{[k]}\|^2 \\
        & + \alpha (\gamma^{[k]} - \gamma^{[k+1]}).
    \end{aligned}
\end{equation}

For $\rho_2 > L$, we have $\frac{N(L - \rho_2)}{2} \|z^{[k+1]} - z^{[k]}\|^2 < 0$, which implies:
\begin{equation}\small
\begin{split}
     E(z^{[k+1]}) = &\sum_{i=1}^N f_i(z^{[k+1]}) + \alpha \gamma^{[k+1]} \\
    \leq &E(z^{[k]}) + \frac{N(L - \rho_2)}{2} \|z^{[k+1]} - z^{[k]}\|^2.
\end{split}
\end{equation}
Therefore, $E(z^{[k+1]}) - E(z^{[k]}) < 0$. Since $E(\cdot)$ is bounded below by the optimal solution from Stage I, the convergence of Stage II follows. \hfill 

\bibliography{ifacconf}             % bib file to produce the bibliography
                                                     % with bibtex (preferred)
                                                   
%\begin{thebibliography}{xx}  % you can also add the bibliography by hand

%\bibitem[Able(1956)]{Abl:56}
%B.C. Able.
%\newblock Nucleic acid content of microscope.
%\newblock \emph{Nature}, 135:\penalty0 7--9, 1956.

%\bibitem[Able et~al.(1954)Able, Tagg, and Rush]{AbTaRu:54}
%B.C. Able, R.A. Tagg, and M.~Rush.
%\newblock Enzyme-catalyzed cellular transanimations.
%\newblock In A.F. Round, editor, \emph{Advances in Enzymology}, volume~2, pages
%  125--247. Academic Press, New York, 3rd edition, 1954.

%\bibitem[Keohane(1958)]{Keo:58}
%R.~Keohane.
%\newblock \emph{Power and Interdependence: World Politics in Transitions}.
%\newblock Little, Brown \& Co., Boston, 1958.

%\bibitem[Powers(1985)]{Pow:85}
%T.~Powers.
%\newblock Is there a way out?
%\newblock \emph{Harpers}, pages 35--47, June 1985.

%\bibitem[Soukhanov(1992)]{Heritage:92}
%A.~H. Soukhanov, editor.
%\newblock \emph{{The American Heritage. Dictionary of the American Language}}.
%\newblock Houghton Mifflin Company, 1992.

%\end{thebibliography}

%\appendix
%\section{A summary of Latin grammar}    % Each appendix must have a short title.
%\section{Some Latin vocabulary}              % Sections and subsections are supported  
                                                                         % in the appendices.
\end{document}